\documentclass{amsart} 

\newtheorem{theorem}{Theorem}[section]
\newtheorem{lemma}[theorem]{Lemma}
\newtheorem{proposition}[theorem]{Proposition}

\theoremstyle{definition}
\newtheorem{definition}[theorem]{Definition}

\newtheorem{remark}[theorem]{Remark}

\allowdisplaybreaks

\newcommand{\Z}{\mathbb{Z}}
\newcommand{\R}{\mathbb{R}}

\newcommand{\e}{\varepsilon}

\makeatletter

\@addtoreset{figure}{section}
\makeatother

\makeatletter

\@addtoreset{table}{section}
\makeatother

\makeatletter
  
  \@addtoreset{equation}{section}
\makeatother

\usepackage{longtable,caption} 
\usepackage[dvipdfmx]{graphicx}
\usepackage{comment} 
\usepackage{lscape}
\usepackage{color}
\usepackage{amsmath,amssymb}
\usepackage{physics}
\usepackage[abs]{overpic}
\usepackage{multirow}

\begin{document}


\title[The intersection polynomials of a long virtual knot II]
{The intersection polynomials of a long virtual knot II: 
Two supporting genera and characterizations}

\author[T. NAKAMURA]{Takuji NAKAMURA}
\address{Faculty of Education, 
University of Yamanashi,
Takeda 4-4-37, Kofu, Yamanashi 400-8510, Japan}
\email{takunakamura@yamanashi.ac.jp}

\author[Y. NAKANISHI]{Yasutaka NAKANISHI}
\address{Department of Mathematics, Kobe University, 
Rokkodai-cho 1-1, Nada-ku, Kobe 657-8501, Japan}
\email{nakanisi@math.kobe-u.ac.jp}

\author[S. SATOH]{Shin SATOH}
\address{Department of Mathematics, Kobe University, 
Rokkodai-cho 1-1, Nada-ku, Kobe 657-8501, Japan}
\email{shin@math.kobe-u.ac.jp}

\author[K. WADA]{Kodai WADA}
\address{Department of Mathematics, Kobe University, 
Rokkodai-cho 1-1, Nada-ku, Kobe 657-8501, Japan}
\email{wada@math.kobe-u.ac.jp}

\makeatletter
\@namedef{subjclassname@2020}{%
  \textup{2020} Mathematics Subject Classification}
\makeatother


\thanks{This work was supported by JSPS KAKENHI Grant Numbers 
JP20K03621, JP22K03287, and JP23K12973.}


\begin{abstract} 
We develop the study of the twelve intersection polynomials of long virtual knots, 
previously introduced in our preceding paper. 
We define two geometric invariants, the $1$- and $2$-supporting genera, 
using two distinct surface realizations. 
These genera yield a natural filtration of the set of long virtual knots, 
and we analyze the behavior of the intersection polynomials 
for long virtual knots with small supporting genera. 
Moreover, we investigate virtual $2$-string tangles, 
analyzing how their sums with long virtual knots affect 
the intersection polynomials through right closures. 
As an application, we provide complete realizability criteria 
for all twelve intersection polynomials.
\end{abstract} 

\maketitle


\section{Introduction}\label{sec1} 

Virtual knot theory, discovered by Kauffman~\cite{Kau}, 
naturally leads to the study of long virtual knots. 
In the preceding paper~\cite{NNSW1}, 
we introduced twelve polynomial invariants 
$F_{ab}(K;t)$, $G_{ab}(K;t)$, and $H_{ab}(K;t)$ 
$(a,b\in\{0,1\})$ of a long virtual knot~$K$, 
collectively called the \emph{intersection polynomials}. 
These invariants are defined via intersection numbers 
of cycles on an oriented closed surface, 
and were shown to be finite-type invariants of degree two 
with respect to crossing changes. 
In~\cite{NNSW1}, we also established their fundamental properties, 
including their behavior under 
symmetries, crossing changes, and concatenation products. 
 
This paper is a sequel to \cite{NNSW1}, 
focusing on the geometric structure and 
realizability of the intersection invariants. 
The first aim of this paper is to introduce 
two geometric invariants of a long virtual knot $K$, 
the \emph{$1$-} and \emph{$2$-supporting genera} 
$sg_{1}(K)$ and $sg_{2}(K)$,  
based on the minimal genus of surface realizations of $K$. 
These genera naturally yield a filtration of the set of long virtual knots 
\[\mathcal{K}_1(0)\subset\mathcal{K}_2(0)\subset
\mathcal{K}_1(1)\subset\mathcal{K}_2(1)\subset
\mathcal{K}_1(2)\subset\cdots, \]
where $\mathcal{K}_i(g)=\{K\mid sg_i(K)\leq g\}$ 
for $i\in\{1,2\}$ and $g\geq 0$. 
We study several properties of the intersection polynomials 
of a long virtual knot with small supporting genera, 
which enable us to prove the strictness of the initial steps of the filtration: 
$\mathcal{K}_1(0)\ne \mathcal{K}_2(0)\ne
\mathcal{K}_1(1)\ne \mathcal{K}_2(1)$. 

The second aim of this paper is to establish a complete characterization 
for each of the twelve intersection polynomials. 
That is, we determine necessary and sufficient conditions 
for a given Laurent polynomial to be realized as each of 
the intersection polynomials of a long virtual knot. 
For example, we prove that a Laurent polynomial $f(t)$ is 
realized as $f(t)=F_{00}(K;t)$ for some long virtual knot $K$ 
if and only if $f(1)=1$ and $f(t)$ is reciprocal. 
We further show that such a knot $K$ can always be chosen 
to satisfy $sg_1(K)\leq 1$.

This paper is organized as follows. 
In Section~\ref{sec2}, 
we recall the definition 
and some fundamental properties of the intersection polynomials 
from~\cite{NNSW1}. 
Section~\ref{sec3} introduces the two supporting genera 
and proves the strictness of the small genus filtration 
(Theorem~\ref{thm38}). 
Sections~\ref{sec4}--\ref{sec6} 
establish essential tools and preliminary characterizations 
needed for the main result. 
Section~\ref{sec4} introduces six polynomial invariants 
of a virtual $2$-string tangle $T$ via its right closure 
and analyzes how the sum $T+K$ affects the intersection polynomials 
(Theorem~\ref{thm43}). 
Section~\ref{sec5} utilizes the right and left closures of $T$ 
to provide a relationship between $F$ and $H$ polynomials 
(Proposition~\ref{prop54}). 
Section~\ref{sec6} is devoted to characterizing  
the writhe polynomial $W_a(K;t)$ under the condition $sg_2(K)=0$ 
(Proposition~\ref{prop62}). 
Finally, in Section~\ref{sec7}, 
we synthesizes these preliminary results to establish complete characterizations of 
the intersection polynomials 
$F_{ab}(K;t)$, $G_{ab}(K;t)$, and $H_{ab}(K;t)$ 
(Theorems~\ref{thm71}, \ref{thm72}, and \ref{thm74}).


\section{Preliminaries}\label{sec2} 

In this section, we review the definitions of 
the writhe and the intersection polynomials, 
together with several of their fundamental properties. 
Refer to \cite{NNSW1} for more details. 

A long virtual knot is presented by a diagram in ${\R}^2$ 
that coincides with the $x$-axis outside a $2$-disk
and is equipped with real and virtual crossings. 
It can also be described by a surface realization 
on an oriented closed surface with a basepoint 
corresponding to $\pm\infty$. 

Let $D$ be a diagram of a long virtual knot $K$, 
$c_1,\dots,c_n$ the real crossings of $D$, 
and $(\Sigma_g,D)$ a surface realization of $D$. 
Here, $\Sigma_g$ denotes an oriented closed surface of genus~$g$. 
For each $i$ with $1\leq i\leq n$, 
traversing $D$ from the basepoint in the positive direction, 
if we pass $c_i$ first over and then under, 
we say that $c_i$ is \emph{of type~$0$}; 
otherwise, it is \emph{of type~$1$}. 
We define two sets by 
\[I_a(D)=\{i\mid \mbox{$c_i$ is of type $a$}\} \ (a=0,1).\]
Smoothing $D$ at $c_i$ yields two cycles on $\Sigma_g$, 
one of which does not contain the basepoint, 
while the other does. 
We denote by $\alpha_i$ the cycle that does not contain the basepoint, 
and by $\beta_i$ the one that does. 

Let $\e_{i}\in\{\pm1\}$ be the sign of $c_{i}$ $(i=1,\dots,n)$. 
For each $a\in\{0,1\}$, the Laurent polynomial 
\[W_a(D;t)=\sum_{i=1}^n 
\varepsilon_i(t^{\alpha_i\cdot\beta_i}-1)\] 
is an invariant of $K$~\cite[Lemma~4.1]{HNNS2}, 
where the dot between cycles denotes 
their intersection number. 
It is called the {\it $a$-writhe polynomial}, 
and is denoted by $W_a(K;t)$. 
Let $\gamma_{D}$ be the cycle on $\Sigma_{g}$ presented by $D$. 
Since $\alpha_{i}+\beta_{i}=\gamma_{D}$ holds for $1\leq i\leq n$, 
we have 
\[
\alpha_{i}\cdot\beta_{i}=\alpha_{i}\cdot(\gamma_{D}-\alpha_{i})
=\alpha_{i}\cdot\gamma_{D}. 
\]
Therefore, in the definition of $W_{a}(K;t)$, 
we may use $\alpha_{i}\cdot\gamma_{D}$ instead of $\alpha_{i}\cdot\beta_{i}$.

For each $a,b\in\{0,1\}$, the Laurent polynomials 
\begin{align*}
F_{ab}(D;t) &=\sum_{i\in I_a(D),\,j\in I_b(D)}
\varepsilon_i\varepsilon_j(t^{\alpha_i\cdot\alpha_j}-1), \\
G_{ab}(D;t) &=\sum_{i\in I_a(D),\,j\in I_b(D)}
\varepsilon_i\varepsilon_j(t^{\alpha_i\cdot\beta_j}-1)
-\omega_b(D)\cdot W_a(K;t), \mbox{ and}\\
H_{ab}(D;t) &=\sum_{i\in I_a(D),\,j\in I_b(D)}
\varepsilon_i\varepsilon_j(t^{\beta_i\cdot\beta_j}-1)
-\omega_a(D)\cdot W_b(K;t)
-\omega_b(D)\cdot W_a(K;t^{-1}) 
\end{align*}
are invariants of $K$~\cite[Theorem~2.3]{NNSW1}, 
which are called the {\it intersection polynomials}, 
and denoted by $F_{ab}(K;t)$, $G_{ab}(K;t)$, 
and $H_{ab}(K;t)$, respectively. 
Here, $\omega_a(D)=\sum_{i\in I_a(D)}\e_i$ 
is the \emph{$a$-writhe} of $D$. 

Any long virtual knot can be presented by a diagram 
with $\omega_0(D)=\omega_1(D)=0$, 
called an \emph{untwisted diagram}. 
For such a diagram, the definitions of $G_{ab}(K;t)$ and $H_{ab}(K;t)$, 
like that of $F_{ab}(K;t)$, take much simple forms. 

A Laurent polynomial $f(t)$ is called \emph{reciprocal} 
if it holds that $f(t)=f(t^{-1})$. 

\begin{lemma}[{\cite[Lemma 2.4 and Theorem 10.3(ii)]{NNSW1}}]\label{lem21}
Any long virtual knot $K$ satisfies the following.
\begin{enumerate}
\setlength{\itemsep}{1mm}
\item 
$F_{01}(K;t)=F_{10}(K;t^{-1})$ and $H_{01}(K;t)=H_{10}(K;t^{-1})$. 
\item 
$F_{00}(K;t)$, $F_{11}(K;t)$, $H_{00}(K;t)$, and $H_{11}(K;t)$ 
are reciprocal. 
\item 
$G_{00}'(K;1)=G_{11}'(K;1)=0$. 
\qed
\end{enumerate}
\end{lemma}

For a diagram $D$ of $K$, 
we define three diagrams $D^{\#}$, $-D$, and $D^{*}$ as follows: 
\begin{itemize}
\item 
$D^{\#}$ is obtained from $D$ by switching the over/under information 
at each real crossing of $D$. 
\item 
$-D$ is obtained from $D$ by reversing the orientation of $D$. 
\item 
$D^{*}$ is obtained from $D$ by an orientation-reversing homeomorphism of $\R^{2}$ 
(or $\Sigma_{g}$). 
\end{itemize}
Let $K^{\#}$, $-K$, and $K^{*}$ denote the long virtual knots 
presented by $D^{\#}$, $-D$, and $D^{*}$, respectively. 
 
\begin{lemma}[{\cite[Theorem 5.1]{NNSW1}}]\label{lem22}
For any $a\in\{0,1\}$, the following hold.
\begin{enumerate}
\setlength{\itemsep}{1mm}
\item 
$W_a(K^\#;t)=-W_{1-a}(K;t)$. 
\item 
$W_a({-K};t)=W_{1-a}(K;t)$. 
\item 
$W_a(K^*;t)=-W_a(K;t^{-1})$. 
\qed
\end{enumerate}
\end{lemma}

\begin{lemma}[{\cite[Theorem 5.2]{NNSW1}}]\label{lem23}
For any $X\in\{F,G,H\}$ and 
$a,b\in\{0,1\}$, the following hold.
\begin{enumerate}
\setlength{\itemsep}{1mm}
\item 
$X_{ab}({K^\#}; t)=X_{ab}(-K;t)=X_{1-a,1-b}(K;t)$. 
\item 
$X_{ab}({K^*};t)=X_{ab}(K;t^{-1})$.
\qed
\end{enumerate}
\end{lemma}

Let $D$ and $D'$ be diagrams of long virtual knots $K$ and $K'$, respectively. 
We denote by $D\circ D'$ the diagram obtained by concatenating $D'$ after $D$. 
The \emph{product} of $K$ and $K'$ is defined as 
the long virtual knot presented by $D\circ D'$, 
and is denoted by $K\circ K'$. 

\begin{lemma}[{\cite[Theorems 6.1 and 6.2]{NNSW1}}]\label{lem24}
For any $a,b\in\{0,1\}$, the following hold.
\begin{enumerate}
\setlength{\itemsep}{1mm}
\item 
$W_a(K\circ K';t)=W_a(K;t)+W_a(K';t)$. 
\item 
$X_{ab}(K\circ K';t)=X_{ab}(K;t)+X_{ab}(K';t)$ for any $X\in\{F,G\}$. 
\item
$H_{ab}(K\circ K';t)=H_{ab}(K;t)+H_{ab}(K';t)$\\
\phantom{$H_{ab}(K\circ K';t)=$}
$+W_a(K;t^{-1})W_b(K';t)+W_a(K';t^{-1})W_b(K;t)$. 
\qed
\end{enumerate}
\end{lemma}

The writhe polynomial, defined independently in~\cite{CG,Kau2013,ST}, 
and the three intersection polynomials introduced in \cite{HNNS1} 
are polynomial invariants of (closed) virtual knots. 
For a long virtual knot $K$, let $\widehat{K}$ denote its closure. 
Then the writhe polynomial $W(\widehat{K};t)$ 
and the first intersection polynomial $I(\widehat{K};t)$ 
can be expressed in terms of those of $K$ as follows.

\begin{lemma}[{\cite[Proposition~10.1]{NNSW1}}]\label{lem25} 
Any long virtual knot $K$ satisfies the following. 
\begin{enumerate}
\setlength{\itemsep}{1mm}
\item
$W(\widehat{K};t)=W_0(K;t)+W_1(K;t^{-1})$. 
\item
$I(\widehat{K};t)=F_{01}(K;t)+G_{00}(K;t)+G_{11}(K;t^{-1})+H_{01}(K;t^{-1})$. 
\qed
\end{enumerate}
\end{lemma}

\section{Two supporting genera}\label{sec3}

The supporting genus of a virtual knot $\kappa$ 
is one of the important notions in virtual knot theory. 
It is defined as the minimal genus for all surface realizations of $\kappa$, 
and is denoted by $sg(\kappa)$. 
If $sg(\kappa)\leq 1$, that is, 
$\kappa$ is presented by a diagram on the torus $\Sigma_{1}$, 
then the following hold. 

\begin{lemma}[\cite{Higa, HNNS4}]\label{lem31} 
If $sg(\kappa)\leq 1$ holds, then $W(\kappa;t)$ and $I(\kappa;t)$ are reciprocal. 
\qed
\end{lemma}

In this section, we introduce two genera of a long virtual knot. 
The first one is defined via a surface realization as follows. 

\begin{definition}\label{def32}
For a long virtual knot $K$, 
\[sg_1(K)=\min\{g\mid \mbox{$(\Sigma_g,D)$ is 
a surface realization of $K$}\}\] 
is called the {\it $1$-supporting genus} of $K$. 
\end{definition}

By definition, $sg_{1}(\widehat{K})\leq sg_{1}(K)$ holds 
for any long virtual knot $K$. 

\begin{lemma}\label{lem33}
If $sg_1(K)=0$ holds, 
then we have the following. 
\begin{enumerate}
\setlength{\itemsep}{1mm}
\item
$W_a(K;t)=0$ for any $a\in\{0,1\}$. 
\item
$X_{ab}(K;t)=0$ 
for any $X\in\{F,G,H\}$ and $a,b\in\{0,1\}$. 
\end{enumerate}
\end{lemma}

\begin{proof}
Since $K$ is presented by a surface realization on the $2$-sphere $\Sigma_{0}$, 
the intersection number of any pair of cycles vanishes. 
Therefore, the conclusion follows. 
\end{proof}

We remark that $sg_1(K)=0$ if and only if $K$ is a long classical knot. 

\begin{lemma}\label{lem34}
If $sg_1(K)\leq 1$ holds, 
then we have the following. 
\begin{enumerate}
\setlength{\itemsep}{1mm}
\item
$W_0(K;t)-W_1(K;t)$ is reciprocal.
\item
$F_{01}(K;t)+G_{00}(K;t)-G_{11}(K;t)-H_{01}(K;t)$ 
is reciprocal. 
\item
$G_{00}(K;t)-G_{01}(K;t)-G_{10}(K;t)+G_{11}(K;t)$ 
is reciprocal. 
\end{enumerate}
\end{lemma}

\begin{proof}
(i) Since the closure of $K$ satisfies $sg(\widehat{K})\leq 1$, 
the writhe polynomial $W(\widehat{K};t)$ of $\widehat{K}$ is 
reciprocal by Lemma~\ref{lem31}. 
Then it follows from Lemma~\ref{lem25}(i) that 
\[W_0(K;t)+W_1(K;t^{-1})=W_0(K;t^{-1})+W_1(K;t).\] 

(ii) By Lemma~\ref{lem31}, 
the first intersection polynomial $I(\widehat{K};t)$ is reciprocal. 
Therefore, it follows from Lemma~\ref{lem25}(ii) that 
\begin{align*}
&F_{01}(K;t)+G_{00}(K;t)+G_{11}(K;t^{-1})+H_{01}(K;t^{-1})\\
&=
F_{01}(K;t^{-1})+G_{00}(K;t^{-1})+G_{11}(K;t)+H_{01}(K;t).
\end{align*}

(iii) Let $K^d$ denote the descending long virtual knot 
associated with $K$ (cf.~\cite[Section~2.2]{FKM}); 
that is, $K^{d}$ is presented a diagram $D$ with $I_{1}(D)=\emptyset$. 
Making a surface realization of $K$ on the torus $\Sigma_{1}$ descending, 
we obtain a surface realization of $K^d$ on $\Sigma_{1}$. 
Hence, we have $sg_1(K^d)\leq 1$, 
and 
\[F_{01}(K^d;t)+G_{00}(K^d;t)-G_{11}(K^d;t)-H_{01}(K^d;t)\]
is reciprocal by (ii). 
On the other hand, it follows from \cite[Corollary~7.4(ii)]{NNSW1} that 
\begin{align*}
&F_{01}(K^d;t)=G_{11}(K^d;t)=H_{01}(K^d;t)=0 \mbox{ and}\\
&G_{00}(K^d;t)=G_{00}(K;t)-G_{01}(K;t)-G_{10}(K;t)+G_{11}(K;t).
\end{align*}
Therefore, the conclusion follows. 
\end{proof}

The second genus of a long virtual knot 
is defined by using a $2$-punctured surface as follows. 
Let $\Sigma_{g,2}$ be a connected, oriented, compact surface 
of genus~$g$ with two boundary components. 
We regard a long virtual knot diagram $D$ 
as a tangle diagram with two boundary points. 
Then we can consider a surface realization of $D$ on $\Sigma_{g,2}$ 
such that the two boundary points of $D$ lie 
on distinct boundary components of $\Sigma_{g,2}$. 
See Figure~\ref{fig:supporting-genus}. 

\begin{figure}[htbp]
  \centering
    \begin{overpic}[]{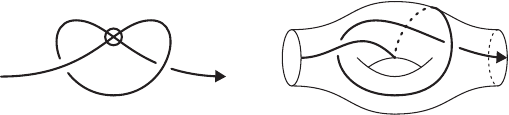} 
      \put(50.5,-12){$D$}
      \put(173,-12){$(\Sigma_{1,2},D)$}
    \end{overpic}
  \vspace{1em}
  \caption{A surface realization $(\Sigma_{1,2},D)$}
  \label{fig:supporting-genus}
\end{figure}

\begin{definition}\label{def35}
For a long virtual knot $K$, 
\[sg_2(K)=\min\{g\mid \mbox{$(\Sigma_{g,2},D)$ is 
a surface realization of $K$}\}\] 
is called the {\it $2$-supporting genus} of $K$. 
\end{definition}

\begin{lemma}\label{lem36}
If $sg_2(K)=0$ holds, 
then we have the following. 
\begin{enumerate}
\setlength{\itemsep}{1mm}
\item
$W_0(K;t)=W_1(K;t)$. 
\item
$F_{ab}(K;t)=G_{ab}(K;t)=0$ 
for any $a,b\in\{0,1\}$. 
\item
$H_{ab}(K;t)=W_0(K;t)W_0(K;t^{-1})$ 
for any $a,b\in\{0,1\}$. 
\end{enumerate}
\end{lemma}

\begin{proof}
(i) By identifying the boundary components of the annulus $\Sigma_{0,2}$, 
we obtain a surface realization of $K$ with a basepoint 
on the torus $\Sigma_{1}$. 
Then it can be deformed into a diagram with no crossings, 
which presents the trivial knot, 
by a finite sequence of crossing changes and Reidemeister moves. 
See Figure~\ref{fig:pf-sg2}. 
Therefore, $W_0(K;t)-W_1(K;t)=0$ holds by \cite[Theorem~7.1]{NNSW1}. 

\begin{figure}[htbp]
  \centering
  \vspace{0.5em}
    \begin{overpic}[]{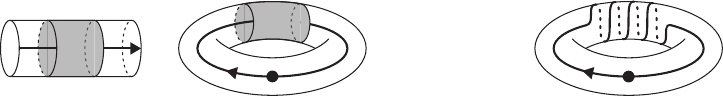}
      \put(31,19.5){$D$}
      \put(17,-12){$(\Sigma_{0,2},D)$}
      \put(128,32.5){$D$}
      \put(115.5,-12){$(\Sigma_{1},D)$}
      \put(180,20.15){$\xrightarrow[\text{Reidemeister moves}]{\text{crossing changes}}$}
      \put(284.5,49){$1$ $2\cdots k$}
    \end{overpic}
  \vspace{1em}
  \caption{Proof of Lemma~\ref{lem36}(i)}
  \label{fig:pf-sg2}
\end{figure}
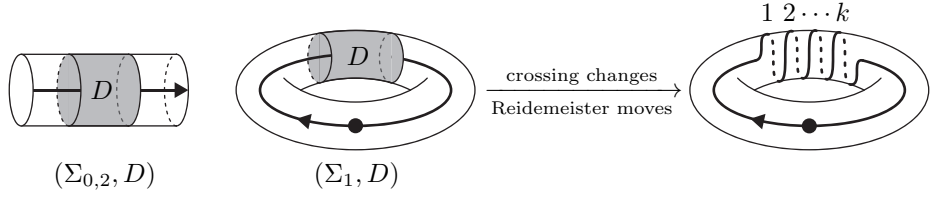

(ii) We may assume that $D$ is untwisted; 
that is, $\omega_0(D)=\omega_1(D)=0$. 
Each cycle $\alpha_i$ satisfies $\alpha_i=k_i \mu$ for some $k_i\in{\Z}$, 
where $\mu$ is a meridian of $\Sigma_{1}$. 
By $\alpha_i\cdot\alpha_j=0$ for any $i$ and $j$, 
we have $F_{ab}(K;t)=0$ 
for any $a,b\in\{0,1\}$. 
On the other hand, since it holds that 
\[\alpha_i\cdot\beta_j=\alpha_i\cdot(\gamma_D-\alpha_j)
=\alpha_i\cdot\gamma_D
=\alpha_i\cdot(\alpha_i+\beta_i)
=\alpha_i\cdot\beta_i,\]
we have 
\begin{align*}
G_{ab}(K;t)&=\sum_{i\in I_a(D),\,j\in I_b(D)}
\varepsilon_i\varepsilon_j(t^{\alpha_i\cdot\beta_j}-1)\\
&=\sum_{i\in I_a(D),\,j\in I_b(D)}
\varepsilon_i\varepsilon_j(t^{\alpha_i\cdot\beta_i}-1)
=\omega_b(D)\cdot W_a(K;t)=0.
\end{align*}

(iii) It holds that 
\[\beta_i\cdot\beta_j=(\gamma_D-\alpha_i)\cdot(\gamma_D-\alpha_j)
=-\alpha_i\cdot\gamma_D+\alpha_j\cdot\gamma_D
=-\alpha_i\cdot\beta_i+\alpha_j\cdot\beta_j.\]
By (i), we have 
\begin{align*}
H_{ab}(K;t)&=\sum_{i\in I_a(D),j\in I_b(D)}
\varepsilon_i\varepsilon_j
(t^{-\alpha_i\cdot\beta_i+\alpha_j\cdot\beta_j}-1)\\
&=
\Biggl(\sum_{i\in I_a(D)}\e_i(t^{-\alpha_i\cdot\beta_i}-1)\Biggr)
\Biggl(\sum_{j\in I_b(D)}\e_j(t^{\alpha_j\cdot\beta_j}-1)\Biggr)\\
&\quad+
\sum_{i\in I_a(D),j\in I_b(D)}\e_i\e_j(t^{-\alpha_i\cdot\beta_i}-1)
+
\sum_{i\in I_a(D),j\in I_b(D)}\e_i\e_j(t^{\alpha_j\cdot\beta_j}-1)\\
&=
W_a(K;t^{-1})W_b(K;t)=W_0(K;t)W_0(K;t^{-1}). 
\end{align*}
\end{proof}

For an integer $g\geq 0$, 
we define two sets of long virtual knots by 
\begin{align*}
\mathcal{K}_1(g)&=\{K\mid sg_1(K) \leq g\} \mbox{ and }\\
\mathcal{K}_2(g)&=\{K\mid sg_2(K) \leq g\}.
\end{align*}

\begin{lemma}\label{lem37}
We have $sg_2(K)\leq sg_1(K)\leq sg_2(K)+1$, 
and hence 
\[\mathcal{K}_1(0)\subset\mathcal{K}_2(0)\subset
\mathcal{K}_1(1)\subset\mathcal{K}_2(1)\subset
\mathcal{K}_1(2)\subset\cdots.\]
\end{lemma}

\begin{proof}
If $K$ has a surface realization $(\Sigma_g,D)$, 
then removing a pair of disks near the basepoint from $\Sigma_g$ 
provides a surface realization $(\Sigma_{g,2},D)$. 
See the left and middle of Figure~\ref{fig:supporting-genus-pf}. 
Therefore, we have $sg_2(K)\leq sg_1(K)$. 

\begin{figure}[htbp]
  \centering
    \begin{overpic}[]{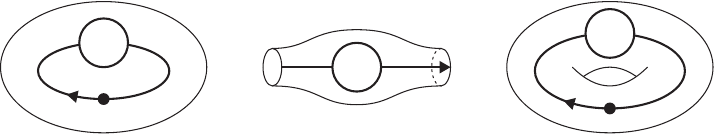} 
      \put(46,40){$D$}
      \put(33,-12){$(\Sigma_{g},D)$}
      \put(168,28.5){$D$}
      \put(153,-12){$(\Sigma_{g,2},D)$}
      \put(289.5,45){$D$}
      \put(273,-12){$(\Sigma_{g+1},D)$}
    \end{overpic}
  \vspace{1em}
  \caption{Proof of Lemma~\ref{lem37}}
  \label{fig:supporting-genus-pf}
\end{figure}

On the other hand, 
if $K$ has a surface realization $(\Sigma_{g,2},D)$, 
then identifying the boundary components of $\Sigma_{g,2}$ 
provides a surface realization $(\Sigma_{g+1},D)$. 
Therefore, we have $sg_1(K)\leq sg_2(K)+1$. 
See the middle and right of the figure. 
\end{proof}

\begin{theorem}\label{thm38}
We have the following. 
\begin{enumerate}
\setlength{\itemsep}{1mm}
\item
$\mathcal{K}_1(0)\ne \mathcal{K}_2(0)$. 
\item
$\mathcal{K}_2(0)\ne\mathcal{K}_1(1)$. 
\item
$\mathcal{K}_1(1)\ne \mathcal{K}_2(1)$. 
\end{enumerate}
Moreover, there are infinitely many long virtual knots 
which realize the difference between the two sets 
in each of {\rm (i)}--{\rm (iii)}. 
\end{theorem}

\begin{proof}
(i) For an integer $n\geq1$, 
we consider the long virtual knot $K_n$ 
presented by a surface realization on the annulus $\Sigma_{0,2}$ 
as shown in Figure~\ref{fig:pf-thm37i}. 
We have $sg_{2}(K_{n})=0$, 
and hence $sg_{1}(K_{n})\leq 1$  
by Lemma~\ref{lem37}. 
Since it can be verified that 
\[
W_0(K_n;t)=n(t-1)\ne 0,
\] 
we have $sg_1(K_n)\ne 0$  by Lemma~\ref{lem33}(i), 
and hence $sg_1(K_n)=1$. 
Therefore, $K_n\in\mathcal{K}_2(0)\setminus\mathcal{K}_1(0)$ holds. 
Moreover, $K_n$'s are all distinct. 

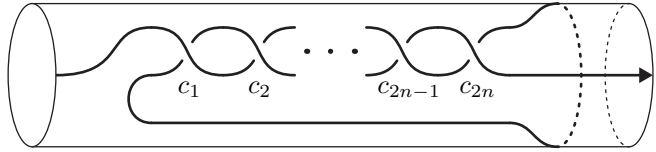
\begin{figure}[htbp]
  \centering
    \begin{overpic}[]{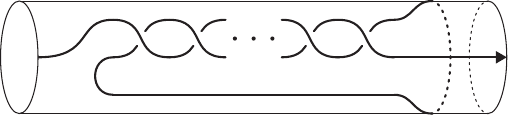}
      \put(64.5,21){$c_{1}$}
      \put(91,21){$c_{2}$}
      \put(140,21){$c_{2n-1}$}
      \put(171,21){$c_{2n}$}
    \end{overpic}
  \caption{A surface realization of $K_n$}
  \label{fig:pf-thm37i}
\end{figure}

(ii) For an integer $n\geq 1$, 
we consider the long virtual knot $K_n'$ 
presented by a surface realization on the torus $\Sigma_{1}$ 
as shown in Figure~\ref{fig:pf-thm37ii}. 
We have $sg_{1}(K_{n}')\leq 1$, 
and hence $sg_{2}(K_{n}')\leq1$ by Lemma~\ref{lem37}. 
Since it can be verified that 
\[
F_{00}(K_n';t)=n(t-2+t^{-1})\ne 0,
\]
we have $sg_2(K_n')=1$ 
by Lemma~\ref{lem36}(ii) and $sg_1(K_n')=1$ by Lemma~\ref{lem37}. 
Therefore, $K_n'\in\mathcal{K}_1(1)\setminus\mathcal{K}_2(0)$ holds. 
Moreover, $K_n'$'s are all distinct. 

\begin{figure}[htbp]
  \centering
    \begin{overpic}[]{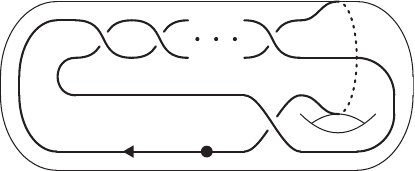}
      \put(47,48){$c_{1}$}
      \put(73,48){$c_{2}$}
      \put(123,48){$c_{2n-1}$}
      \put(126,5){$c_{2n}$}
    \end{overpic}
  \caption{A surface realization of $K_n'$}
  \label{fig:pf-thm37ii}
\end{figure}
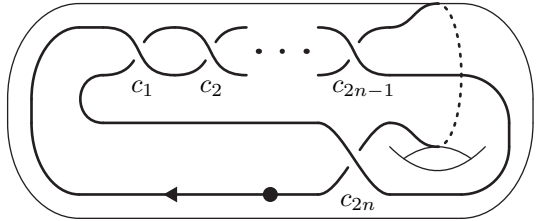

(iii) For an integer $n\geq1$, 
we consider the long virtual knot $K_n''$ 
presented by a surface realization on the $2$-punctured torus $\Sigma_{1,2}$ 
as shown in Figure~\ref{fig:pf-thm37iii}. 
We have $sg_{2}(K_{n}'')\leq 1$, and hence, $sg_{1}(K_{n}'')\leq2$ 
by Lemma~\ref{lem37}. 
Since 
\[
W_0(K_n'';t)-W_1(K_n'';t)=\qty((n-1)t^2+2t-(n+1))-n(t^2-1)=-t^2+2t-1
\]
is not reciprocal, 
we have $sg_1(K_n'')=2$ by Lemma~\ref{lem34}(i) 
and $sg_2(K_n'')=1$ by Lemma~\ref{lem37}. 
Therefore, $K_n''\in\mathcal{K}_2(1)\setminus\mathcal{K}_1(1)$ holds. 
Moreover, $K_n''$'s are all distinct. 
\end{proof}

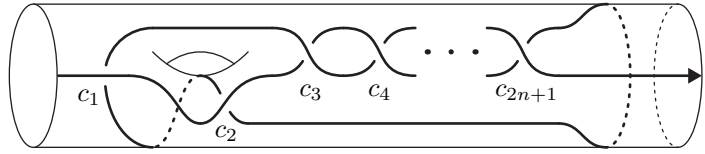
\begin{figure}[htbp]
  \centering
    \begin{overpic}[]{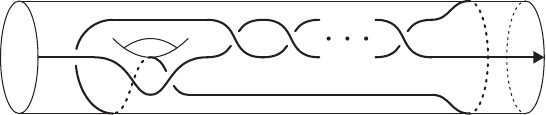}
      \put(26,18){$c_{1}$}
      \put(78,4.5){$c_{2}$}
      \put(110,20){$c_{3}$}
      \put(136,20){$c_{4}$}
      \put(184,20){$c_{2n+1}$}
    \end{overpic}
  \caption{A surface realization of $K_n''$}
  \label{fig:pf-thm37iii}
\end{figure}

\begin{remark}\label{rem39} 
By definition, 
the $1$- and $2$-supporting genera of the product 
of long virtual knots $K_{1}$ and $K_{2}$ satisfy 
\begin{enumerate}
\setlength{\itemsep}{1mm}
\item
$sg_{1}(K_{1}\circ K_{2})\leq sg_{1}(K_{1})+sg_{1}(K_{2})$ and 
\item
$sg_{2}(K_{1}\circ K_{2})\leq sg_{2}(K_{1})+sg_{2}(K_{2})$.  
\end{enumerate} 
Moreover, we can construct a pair of long virtual knots 
which does not attain the equality in (i) as follows. 
Suppose that $K_1$ and $K_2$ are nonclassical long virtual knots 
with $sg_2(K_1)=sg_2(K_2)=0$, and that $K_1\circ K_2$ is also nonclassical. 
Since each of $K_1$, $K_2$, and $K_1\circ K_2$ 
has a surface realization on the torus $\Sigma_{1}$, 
we have 
\[sg_1(K_1)=sg_1(K_2)=sg_1(K_1\circ K_2)=1.\]
\end{remark}

\section{Invariants of a virtual tangle}\label{sec4}

We consider an oriented virtual tangle diagram $E$ 
consisting of two strings $A$ and $B$, 
each connecting two of the four endpoints 
as shown on the left of Figure~\ref{fig:tangle-closure}. 
A {\it virtual $2$-string tangle} $T$ is defined as 
an equivalence class of such diagrams $E$ 
under Reidemeister moves keeping the endpoints fixed. 
Let $R(E)$ denote the long virtual knot diagram 
obtained from $E$ by closing the pair of the right endpoints, 
as shown on the right of the figure. 
The {\it right closure} of $T$ 
is the long virtual knot presented by $R(E)$, 
and is denoted by $R(T)$. 

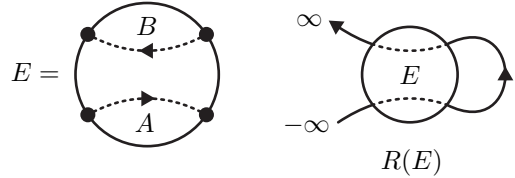
\begin{figure}[htbp]
  \centering
    \begin{overpic}[]{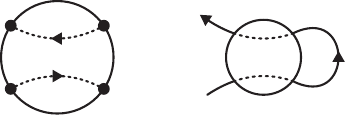}
      \put(-24,25.1){$E=$}
      \put(23.5,42.5){$B$}
      \put(23,6){$A$}
      
      \put(84,46){$\infty$}
      \put(79,6){$-\infty$}
      \put(123,24){$E$}
      \put(116,-8){$R(E)$}
    \end{overpic}
  \vspace{1em}
  \caption{A tangle diagram and its right closure}
  \label{fig:tangle-closure}
\end{figure}

Any invariant of a long virtual knot 
induces an invariant of a virtual $2$-string tangle 
by using the right closure. 
In particular, the writhe polynomials 
$W_0(R(T);t)$ and $W_1(R(T);t)$ 
are regarded as invariants of $T$. 
These invariants can be elaborated as follows. 
Let $c_1,\cdots,c_m$ be the real crossings of $E$. 
For each $i$ with $1\leq i\leq m$, we classify $c_i$ as follows; 
\begin{itemize}
\item
$c_i$ is {\it of type $(A,A;0)$} 
if both intersecting paths at $c_i$ belong to the string $A$, 
and we pass $c_i$ first over and 
then under along $A$, 
\item
$c_i$ is {\it of type $(A,A;1)$} 
if both paths belong to $A$, 
and we pass $c_i$ first under and then over along $A$, 
\item
$c_i$ is {\it of type $(B,B;0)$} 
if both paths belong to $B$, 
and we pass $c_i$ first over and then under along $B$, 
\item
$c_i$ is {\it of type $(B,B;1)$} 
if the paths at $c_i$ both belong to $B$, 
and we pass $c_i$ first going under and 
then over along $B$, 
\item
$c_i$ is {\it of type $(A,B)$} 
if the over-path at $c_i$ belongs to $A$ and 
the under-path to $B$, and 
\item
$c_i$ is {\it of type $(B,A)$} 
if the over-path belongs to $B$ and 
the under-path to $A$. 
\end{itemize}

We define six sets by  
\begin{center}
\begin{tabular}{ll}
$J_0^A(E)=\{i\mid \mbox{$c_i$ is of type $(A,A;0)$}\}$, & 
$J_1^A(E)=\{i\mid \mbox{$c_i$ is of type $(A,A;1)$}\}$, \smallskip\\
$J_0^B(E)=\{i\mid \mbox{$c_i$ is of type $(B,B;0)$}\}$, & 
$J_1^B(E)=\{i\mid \mbox{$c_i$ is of type $(B,B;1)$}\}$, \smallskip\\
$J_0'(E)=\{i\mid \mbox{$c_i$ is of type $(A,B)$}\}$, & 
$J_1'(E)=\{i\mid \mbox{$c_i$ is of type $(B,A)$}\}$. 
\end{tabular}
\end{center}
Moreover, we put 
\[J_0(E)=J_0^A(E)\cup J_0^B(E) \mbox{ and }
J_1(E)=J_1^A(E)\cup J_1^B(E).\] 
For any $a\in\{0,1\}$, it follows by definition that 
\[I_a(R(E))=J_a(E)\cup J_a'(E).\]

We take a surface realization 
$(\Sigma_g,R(E))$ of the right closure $R(E)$. 
Let $\alpha_i$ and $\beta_i$ 
be the cycles defined at $c_i$ associated with 
$(\Sigma_g,R(E))$. 
For $a\in\{0,1\}$ and $X\in\{A,B\}$, 
we define six Laurent polynomials by 
\[U_a^X(E;t)=\sum_{i\in J_a^X(E)}\e_i(t^{\alpha_i\cdot\beta_i}-1) \mbox{ and }
V_a(E;t)=\sum_{i\in J_a'(E)}\e_it^{\alpha_i\cdot\beta_i}.\]

\begin{lemma}\label{lem41} 
The Laurent polynomials 
\[U_0^A(E;t), \ U_0^B(E;t), \ 
U_1^A(E;t), \ U_1^B(E;t),\ V_0(E;t), \mbox{ and } V_1(E;t) 
\] 
are invariants of $T$. 
\end{lemma}

\begin{proof}
For each $a\in\{0,1\}$ and $X\in\{A,B\}$, 
the invariance of the Laurent polynomials 
$U_{a}^{X}(E;t)$ and $V_{a}(E;t)$ 
can be proved in a similar manner to 
the proof of \cite[Lemma 4.1]{HNNS2}. 
\end{proof}

We may replace the letter $E$ with $T$ 
in the invariants of Lemma~\ref{lem41}. 
For $a\in\{0,1\}$, we put 
\[\lambda_a(T)=V_a(T;1)=\sum_{i\in J_a'(E)}\e_i.\]  
The integers $\lambda_0(T)$ and $\lambda_1(T)$ 
are called the {\it linking numbers} of $T$, 
and $U_a^X(T;t)$ and $V_a(T;t)$ 
are the {\it writhe polynomials} of $T$. 
We also put 
\[U_a(T;t)=U_a^A(T;t)+U_a^B(T;t)
=\sum_{i\in J_a(E)}\e_i(t^{\alpha_i\cdot\beta_i}-1).
\]

\begin{lemma}\label{lem42}
For any $a\in\{0,1\}$, we have 
\[W_a(R(T);t)=
U_a(T;t)+V_a(T;t)-\lambda_a(T).\]
\end{lemma}

\begin{proof}
Since $I_a(R(E))=J_a(E)\cup J_a'(E)$, 
we have 
\begin{align*}
W_a(R(T);t)&=
\sum_{i\in J_a(E)}\e_i(t^{\alpha_i\cdot\beta_i}-1)
+\sum_{i\in J_a'(E)}\e_i(t^{\alpha_i\cdot\beta_i}-1)\\
&=U_a(T;t)+V_a(T;t)-\lambda_a(T). 
\end{align*}
\end{proof}

Inspired by Conway's tangle sum, 
we define a sum of a virtual $2$-string tangle $T$ 
and a long virtual knot $K$ as follows. 
Let $E$ and $D$ be diagrams of $T$ and $K$, respectively. 
We form a long virtual knot diagram, denoted by $E+D$, 
by connecting the endpoints of $E$ and $D$ 
as shown in Figure~\ref{fig:generalized-product}. 
The {\it sum} of $T$ and $K$, 
written $T+K$, is the long virtual knot 
presented by $E+D$.

\begin{figure}[htbp]
  \centering
    \begin{overpic}[]{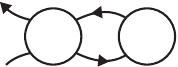}
      \put(-13,29){$\infty$}
      \put(-18,-2){$-\infty$}
      \put(21,12){$E$}
      \put(67,12){$D$}
    \end{overpic}
  \caption{A long virtual knot digram $E+D$ }
  \label{fig:generalized-product}
\end{figure}
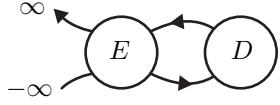

\begin{theorem}\label{thm43}
For any $a,b\in\{0,1\}$, 
we have the following. 
\begin{enumerate}
\setlength{\itemsep}{1mm}
\item
$W_a(T+K;t)=W_a(R(T);t)+W_a(K;t)$. 
\item 
$F_{ab}(T+K;t)=F_{ab}(R(T);t)+F_{ab}(K;t)$
\smallskip\\
\hphantom{$F_{ab}(T+K;t)=$}
$+\lambda_b(T)\cdot W_a(K;t)+\lambda_a(T)\cdot W_b(K;t^{-1})$. 

\item 
$G_{ab}(T+K;t)=G_{ab}(R(T);t)+G_{ab}(K;t)$
\smallskip\\
\hphantom{$G_{ab}(T+K;t)=$}
$-\lambda_b(T)\cdot W_a(K;t)
+V_a(T;t)\cdot W_b(K;t)$. 

\item 
$H_{ab}(T+K;t)=H_{ab}(R(T);t)+H_{ab}(K;t)$
\smallskip\\
\hphantom{$H_{ab}(T+K;t)=$}
$+(U_b(T;t)-\lambda_b(T))\cdot W_a(K;t^{-1})$
\smallskip\\
\hphantom{$G_{ab}(T+K;t)=$}
$+(U_a(T;t^{-1})-\lambda_a(T))\cdot W_b(K;t)$.
\end{enumerate}
\end{theorem}

We take a surface realization $(\Sigma_g,R(E))$ 
of $R(E)$. 
Let $c_1,\dots,c_m$ be the real crossings of $R(E)$, 
and $\alpha_i$ and $\beta_i$ $(1\leq i\leq m)$ 
the cycles at $c_i$ on $\Sigma_g$. 
We may assume that 
\[J_0(E)\cup J_1(E)=\{1,\dots,\ell\}
\mbox{ and }
J_0'(E)\cup J_1'(E)=\{\ell+1,\dots,m\}
\]
for some $\ell$. 
We also take a surface realization $(\Sigma_h,D)$ of $D$. 
Let $c_{m+1},\dots,c_n$ be the real crossings of $D$, 
and $\alpha_i$ and $\beta_i$ $(m+1\leq i\leq n)$ the cycles at $c_i$ on $\Sigma_h$. 

Taking a connected sum of 
$(\Sigma_g,R(E))$ and $(\Sigma_h,D)$
yields a surface realization 
$(\Sigma_{g+h},E+D)$ of $E+D$  
as shown in Figure~\ref{fig:surface-realization-generalized-product}. 
Let $\alpha_i'$ and $\beta_i'$ $(1\leq i\leq n)$ 
be the cycles at $c_i$ on $\Sigma_{g+h}$. 
We also regard the cycles $\alpha_i$ and $\beta_i$ 
as lying on $\Sigma_{g+h}$ ($1\leq i\leq n$). 
To prove Theorem~\ref{thm43}, 
we prepare the following lemma. 

\begin{figure}[htbp]
  \centering
    \begin{overpic}[]{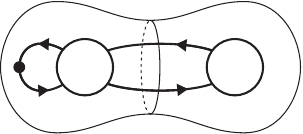}
      \put(37,29.5){$E$}
      \put(109,29.5){$D$}
    \end{overpic}
  \caption{The surface realization $(\Sigma_{g+h},E+ D)$}
  \label{fig:surface-realization-generalized-product}
\end{figure}

\begin{lemma}\label{lem44}
The intersection numbers $\alpha_i'\cdot\alpha_j'$, 
$\alpha_i'\cdot\beta_j'$, and $\beta_i'\cdot\beta_j'$ 
are given as shown in {\rm Tables~\ref{table41}--\ref{table43}}. 
\end{lemma}

\begin{table}[htb]
\caption{The intersection numbers $\alpha_i'\cdot\alpha_j'$}
\label{table41}
\renewcommand{\arraystretch}{1.3}
\begin{tabular}{|c||c|c||c|}
\hline
 & $1\leq j\leq\ell$ & $\ell+1\leq j\leq m$ & $m+1\leq j\leq n$ \\
\hline\hline
$1\leq i\leq \ell$ & \multicolumn{2}{c||}{\multirow{2}{*}{$\alpha_i\cdot\alpha_j$}}
 & $0$ \\ \cline{1-1} \cline{4-4}
$\ell+1\leq i\leq m$ & \multicolumn{1}{c}{} & \multicolumn{1}{c||}{} &  $-\alpha_j\cdot\beta_j$ \\
\hline\hline
$m+1\leq i\leq n$ & $0$ & $\alpha_i\cdot\beta_i$ 
& $\alpha_i\cdot\alpha_j$ \\
\hline
\end{tabular}
\end{table}

\begin{table}[htb]
\caption{The intersection numbers $\alpha_i'\cdot\beta_j'$}
\label{table42}
\renewcommand{\arraystretch}{1.3}
\begin{tabular}{|c||c|c||c|}
\hline
 & $1\leq j\leq\ell$ & $\ell+1\leq j\leq m$ & $m+1\leq j\leq n$ \\
\hline\hline
$1\leq i\leq \ell$ & \multicolumn{2}{c||}{\multirow{2}{*}{$\alpha_i\cdot\beta_j$}}
 & $\alpha_i\cdot\beta_i$ \\ \cline{1-1} \cline{4-4}
$\ell+1\leq i\leq m$ & \multicolumn{1}{c}{} 
& \multicolumn{1}{c||}{} &  $\alpha_i\cdot\beta_i+\alpha_j\cdot\beta_j$ \\
\hline\hline
$m+1\leq i\leq n$ & $\alpha_i\cdot\beta_i$ & $0$ 
&  $\alpha_i\cdot\beta_j$ \\
\hline
\end{tabular}
\end{table}

\begin{table}[htb]
\caption{The intersection numbers $\beta_i'\cdot\beta_j'$}
\label{table43}
\renewcommand{\arraystretch}{1.3}
\begin{tabular}{|c||c|c||c|}
\hline
 & $1\leq j\leq\ell$ & $\ell+1\leq j\leq m$ & $m+1\leq j\leq n$ \\
\hline\hline
$1\leq i\leq \ell$ & \multicolumn{2}{c||}{\multirow{2}{*}{$\beta_i\cdot\beta_j$}}
 & $-\alpha_i\cdot\beta_i+\alpha_j\cdot\beta_j$ \\ \cline{1-1} \cline{4-4}
$\ell+1\leq i\leq m$ & \multicolumn{1}{c}{} 
& \multicolumn{1}{c||}{} &   $-\alpha_i\cdot\beta_i$\\
\hline\hline
$m+1\leq i\leq n$ & $-\alpha_i\cdot\beta_i+\alpha_j\cdot\beta_j$ 
& $\alpha_j\cdot\beta_j$ & $\beta_i\cdot\beta_j$ \\
\hline
\end{tabular}
\end{table}

\begin{proof}
It follows by definition that 
\begin{enumerate}
\item
$\alpha_i'=\alpha_i$ and 
$\beta_i'=\beta_i+\gamma_D$ 
for $1\leq i\leq \ell$, 
\item
$\alpha_i'=\alpha_i+\gamma_D$ and 
$\beta_i'=\beta_i$ for $\ell+1\leq i\leq m$, and 
\item
$\alpha_i'=\alpha_i$ and 
$\beta_i'=\beta_i+\gamma_{R(E)}$ 
for $m+1\leq i\leq n$. 
\end{enumerate}

We consider only the case of the intersection numbers 
$\alpha_i'\cdot\beta_j'$; 
the other cases are treated analogously. 

\underline{$1\leq i,j\leq m$.} 
Since $\alpha_i\cdot\gamma_D=\beta_i\cdot\gamma_D=0$, 
it follows from (i) and (ii) that 
\[\alpha_i'\cdot\beta_j'=
(\mbox{$\alpha_i$ or $\alpha_i+\gamma_D$})\cdot
(\mbox{$\beta_j+\gamma_D$ or $\beta_j$})
=\alpha_i\cdot\beta_j.\]

\underline{$1\leq i\leq \ell$ and $m+1\leq j\leq n$.} 
Since $\alpha_i\cdot\beta_j=0$, 
it follows from (i) and (iii) that 
\[\alpha_i'\cdot\beta_j'=
\alpha_i\cdot(\beta_j+\gamma_{R(E)})=\alpha_i\cdot\gamma_{R(E)}
=\alpha_i\cdot(\alpha_i+\beta_i)=\alpha_i\cdot\beta_i.
\]

\underline{$\ell+1\leq i\leq m$ and $m+1\leq j\leq n$.} 
Since $\alpha_i\cdot\beta_j=
\gamma_D\cdot\gamma_{R(E)}=0$, 
it follows from (ii) and (iii) that 
\[\alpha_i'\cdot\beta_j'=
(\alpha_i+\gamma_D)\cdot(\beta_j+\gamma_{R(E)})
=\alpha_i\cdot\gamma_{R(E)}+\gamma_D\cdot\beta_j
=\alpha_i\cdot\beta_i+\alpha_j\cdot\beta_j.\]

\underline{$m+1\leq i\leq n$ and $1\leq j\leq \ell$.} 
Since $\alpha_i\cdot\beta_j=0$, 
it follows from (i) and (iii) that 
\[\alpha_i'\cdot\beta_j'=
\alpha_i\cdot(\beta_j+\gamma_D)=\alpha_i\cdot\gamma_D
=\alpha_i\cdot\beta_i.\]

\underline{$m+1\leq i\leq n$ and $\ell+1\leq j\leq m$.} 
It follows from (ii) and (iii) that 
\[\alpha_i'\cdot\beta_j'=
\alpha_i\cdot\beta_j=0.\]

\underline{$m+1\leq i,j\leq n$.} 
Since $\alpha_i\cdot\gamma_{R(E)}=0$, 
it follows from (iii) that 
\[\alpha_i'\cdot\beta_j'=
\alpha_i\cdot(\beta_j+\gamma_{R(E)})=\alpha_i\cdot\beta_j.\]
\end{proof}

\begin{proof}[Proof of {\rm Theorem~\ref{thm43}}.]
(i) Since $I_a(E+D)=I_a(R(E))\cup I_a(D)$, 
it follows from Table~\ref{table42} that 
\begin{align*}
W_a(T+K;t)
&=\sum_{i\in I_a(E+D)} \varepsilon_i(t^{\alpha_i'\cdot\beta_i'}-1)\\
&=\sum_{i\in I_a(R(E))}
\varepsilon_i(t^{\alpha_i\cdot \beta_i}-1)+
\sum_{i\in I_a(D)}
\varepsilon_i(t^{\alpha_i\cdot \beta_i}-1)\\
&=W_a(R(T);t)+W_a(K;t).
\end{align*}

(ii) 
By Table~\ref{table41}, we have 
\begin{align*}
&F_{ab}(T+K;t)-F_{ab}(R(T);t)-F_{ab}(K;t)\\
&=
\sum_{\substack{i\in J_a'(E)\\j\in I_b(D)}}
\e_i\e_j(t^{-\alpha_j\cdot\beta_j}-1) 
+\sum_{\substack{i\in I_a(D)\\ j\in J_b'(E)}}
\e_i\e_j(t^{\alpha_i\cdot\beta_i}-1) \\
&=
\Biggl(\sum_{i\in J_a'(E)}\e_i\Biggr)
\Biggl(\sum_{j\in I_b(D)}\e_j(t^{-\alpha_j\cdot\beta_j}-1)\Biggr)
+\Biggl(\sum_{j\in J_b'(E)}\e_j\Biggr)
\Biggl(\sum_{i\in I_a(D)}\e_i(t^{\alpha_i\cdot\beta_i}-1)\Biggr)\\
&=
\lambda_a(T)\cdot W_b(K;t^{-1})
+\lambda_b(T)\cdot W_a(K;t). 
\end{align*}

(iii) By applying Reidemeister moves I to $E$ and $D$ if necessary, 
we may assume that $R(E)$ and $D$ are untwisted. 
By Table~\ref{table42}, 
we have 
\begin{align*}
&G_{ab}(T+K;t)-G_{ab}(R(T);t)-G_{ab}(K;t)\\
&=g_{ab}(E+D;t)-g_{ab}(R(E);t)-g_{ab}(D;t)\\
&=
\sum_{\substack{i\in J_a(E)\\j\in I_b(D)}}
\e_i\e_j(t^{\alpha_i\cdot\beta_i}-1)
+\sum_{\substack{i\in J_a'(E)\\j\in I_b(D)}}
\e_i\e_j(t^{\alpha_i\cdot\beta_i+\alpha_j\cdot\beta_j}-1) 
\\
&\quad
+\sum_{\substack{i\in I_a(D)\\ j\in J_b(E)}}
\e_i\e_j(t^{\alpha_i\cdot\beta_i}-1). 
\end{align*}

Since $D$ is untwisted, we have $\omega_b(D)=0$. 
Therefore, the first sum is equal to 
\[\Biggl(\sum_{j\in I_b(D)}\e_j\Biggr) 
\Biggl(\sum_{i\in J_a(E)}\e_i(t^{\alpha_i\cdot\beta_i}-1)\Biggr)
=\omega_b(D)\cdot U_a(T;t)=0.\]
By the same reason, the second sum is equal to 
\begin{align*}
&\Biggl(\sum_{i\in J_a'(E)}\e_it^{\alpha_i\cdot\beta_i}\Biggr)
\Biggl(\sum_{j\in I_b(D)}\e_j(t^{\alpha_j\cdot\beta_j}-1)\Biggr)
+\Biggl(\sum_{j\in I_b(D)}\e_j\Biggr)
\Biggl(\sum_{i\in J_a'(E)}\e_i(t^{\alpha_i\cdot\beta_i}-1)\Biggr)\\
&=V_a(T;t)\cdot W_b(K;t)
+\omega_b(D)\cdot (V_a(T;t)-\lambda_a(T))\\
&=V_a(T;t)\cdot W_b(K;t). 
\end{align*}
Moreover, since $R(E)$ is untwisted, 
we have 
$\sum_{j\in J_b(E)}\e_j+\lambda_b(T)=\omega_b(R(E))=0$. 
Therefore, 
the third sum is equal to 
\[\Biggl(\sum_{j\in J_b(E)}\e_j\Biggr)
\Biggl(\sum_{i\in I_a(D)}\e_i(t^{\alpha_i\cdot\beta_i}-1)\Biggr)
=-\lambda_b(T)\cdot W_a(K;t).\]
Therefore, the required equation follows.

(iv) We may assume that $R(E)$ and $D$ are untwisted. 
By Table~\ref{table43}, we have  
\begin{align*}
&H_{ab}(T+K;t)-H_{ab}(R(T);t)-H_{ab}(K;t)\\
&=h_{ab}(E+D;t)-h_{ab}(R(E);t)-h_{ab}(D;t)\\
&=
\sum_{\substack{i\in J_a(E)\\j\in I_b(D)}}
\e_i\e_j(t^{-\alpha_i\cdot\beta_i+\alpha_j\cdot\beta_j}-1)
+\sum_{\substack{i\in J_a'(E)\\j\in I_b(D)}}
\e_i\e_j(t^{-\alpha_i\cdot\beta_i}-1)\\
&\quad
+\sum_{\substack{i\in I_a(D)\\j\in J_b(E)}}
\e_i\e_j(t^{-\alpha_i\cdot\beta_i+\alpha_j\cdot\beta_j}-1)
+\sum_{\substack{i\in I_a(D)\\j\in J_b'(E)}}
\e_i\e_j(t^{\alpha_j\cdot\beta_j}-1).
\end{align*}

Since $R(E)$ and $D$ are untwisted, the first sum is equal to 
\begin{align*}
&\Biggl(\sum_{i\in J_a(E)}\e_i(t^{-\alpha_i\cdot\beta_i}-1)\Biggr)
\Biggl(\sum_{j\in I_b(D)}\e_j(t^{\alpha_j\cdot\beta_j}-1)\Biggr)\\
&\quad +\Biggl(\sum_{j\in I_b(D)}\e_j\Biggr)
\Biggl(\sum_{i\in J_a(E)}\e_i(t^{-\alpha_i\cdot\beta_i}-1)\Biggr)
+\Biggl(\sum_{i\in J_a(E)}\e_i\Biggr)
\Biggl(\sum_{j\in I_b(D)}\e_j(t^{\alpha_j\cdot\beta_j}-1)\Biggr)\\
&=U_a(T;t^{-1})\cdot W_b(K;t)+\omega_b(D)\cdot U_a(T;t^{-1})
-\lambda_a(T)\cdot W_b(K;t)\\
&=(U_a(T;t^{-1})-\lambda_a(T))\cdot W_b(K;t). 
\end{align*}

Since the third sum is obtained from the first sum 
by replacing $t$ with $t^{-1}$ and by exchanging $a$ and $b$, 
it is equal to 
$(U_b(T;t)-\lambda_b(T))\cdot W_a(K;t^{-1})$. 

On the other hand, since $D$ is untwisted, 
the second sum is equal to 
\[\Biggl(\sum_{j\in I_b(D)}\e_j\Biggr)
\Biggl(\sum_{i\in J_a'(E)}\e_i
(t^{-\alpha_i\cdot\beta_i}-1)\Biggr)
=\omega_b(D)\cdot (V_a(T;t^{-1})-\lambda_a(T))=0.\]
Similarly, the fourth sum is equal to zero. 
Therefore, the conclusion follows. 
\end{proof}

\section{The left closure of a tangle}\label{sec5}

For a diagram $E$ of a virtual $2$-string tangle $T$, 
let $L(E)$ denote the long virtual knot diagram 
obtained from $E$ by closing the pair of the left endpoints, 
as shown in Figure~\ref{fig:left-closure}. 
The {\it left closure} of $T$ is the long virtual knot 
presented by $L(E)$, and is denoted by $L(T)$. 

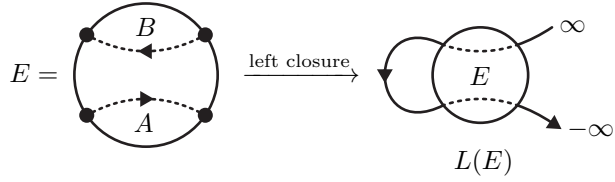
\begin{figure}[htbp]
  \centering
    \begin{overpic}[]{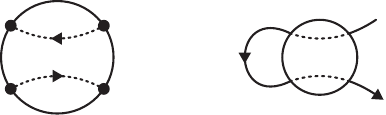}
      \put(-24,25.1){$E=$}
      \put(23.5,42.5){$B$}
      \put(23,6){$A$}
      \put(64,25.1){$\xrightarrow{\text{left closure}}$}
      \put(184,44){$\infty$}
      \put(187,4.5){$-\infty$}
      \put(150,24){$E$}
      \put(144,-8){$L(E)$}
    \end{overpic}
  \vspace{1em}
  \caption{The left closure}
  \label{fig:left-closure}
\end{figure}

\begin{definition}\label{def51}
A virtual $2$-string tangle $T$ is called {\it simply linked} 
if it is presented by a diagram $E$ 
with \[J_0(E)=J_1(E)=\emptyset,\] 
which means that any real crossing of $E$ occurs between $A$ and $B$. 
\end{definition}

\begin{lemma}\label{lem52}
For any long virtual knot $K$, 
there is a virtual $2$-string tangle $T$ 
such that 
\begin{enumerate}
\item
$T$ is simply linked, 
\item
$\lambda_0(T)=\lambda_1(T)=0$, 
\item
$R(T)=K$, 
\item
$\widehat{L(T)}=\widehat{K}$, and 
\item
$sg_1(L(T))\leq sg_1(K)$.
\end{enumerate}
\end{lemma}

\begin{proof}
Let $(\Sigma_g,D)$ be a surface realization of $K$ 
with $sg_1(K)=g$, 
and $p$ the basepoint on $D$. 
By applying Reidemeister moves I if necessary, 
we may assume that $D$ is untwisted; 
that is, $\omega_0(D)=\omega_1(D)=0$.

First, we choose a point $q$ on $D$ 
ahead of $p$ with respect to the orientation of $D$. 
Let $A$ and $B$ be the subarcs of $D$ 
such that $A$ runs from $p$ to $q$ 
and $B$ runs from $q$ to $p$. 
Note that there is no real crossing between two paths of $A$. 
While traveling $B$, 
let $c$ be the crossing between two paths of $B$ closest to $q$. 
Next, we move $c$ to a position behind $q$ by 
a finite sequence of Reidemeister moves 
as shown in Figure~\ref{fig:pf-lem112}, 
where $B$ is indicated by a thick line. 
This procedure reduces the number of crossings between two paths of $B$, 
introduces no new real crossing between two paths of $A$, 
and preserves both the $0$- and $1$-writhes of $D$. 
By repeating this process, 
we finally obtain an untwisted diagram $D$ on $\Sigma_g$ 
with two points $p$ and $q$ 
such that any crossing occurs between $A$ and $B$. 

\begin{figure}[htbp]
  \centering
    \medskip
    \begin{overpic}[]{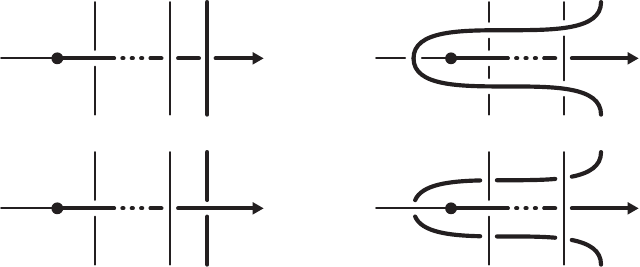}
      \put(-8,99){$p$}
      \put(8,87){$A$}
      \put(25.3,107){$q$}
      \put(41.5,132){$A$}
      \put(78,132){$A$}
      \put(96,132){$B$}
      \put(108.5,87){$B$}
      \put(105,105.5){$c$}
      \put(130,99){$p$}
      \put(146,95.2){{\Huge $\sim$}}
      \put(172.5,99){$p$}
      \put(185,87){$A$}
      \put(191.5,105.5){$c$}
      \put(215,106.5){$q$}
      \put(231,132){$A$}
      \put(267.5,132){$A$}
      \put(286,132){$B$}
      \put(289,87){$B$}
      \put(310.5,99){$p$}
      \put(-8,27){$p$}
      \put(8,15){$A$}
      \put(25.3,35){$q$}
      \put(40.5,-12){$A$}
      \put(77,-12){$A$}
      \put(95,-12){$B$}
      \put(108.5,15){$B$}
      \put(105,34){$c$}
      \put(130,27){$p$}
      \put(146,23){{\Huge $\sim$}}
      \put(172.5,27){$p$}
      \put(185,15){$A$}
      \put(191.5,33){$c$}
      \put(215,34){$q$}
      \put(230,-12){$A$}
      \put(266.5,-12){$A$}
      \put(285,-12){$B$}
      \put(289,15){$B$}
      \put(310.5,27){$p$}
    \end{overpic}
  \vspace{1em}
  \caption{Moving $c$ to a position behind $q$}
  \label{fig:pf-lem112}
\end{figure}
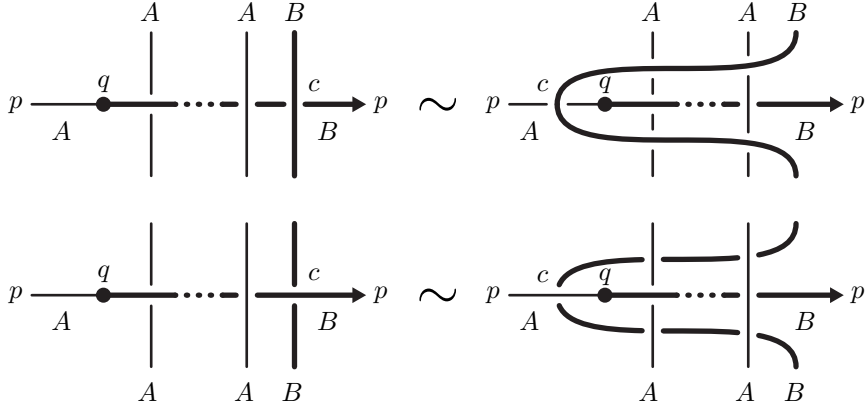

Let $E$ be a virtual $2$-string tangle diagram 
obtained from the resulting diagram $D$ 
by cutting it at $p$ and $q$, 
and $T$ the tangle presented by $E$. 
Since there is no crossing between two paths of $A$ 
or two paths of $B$, 
we have $J_0(E)=J_1(E)=\emptyset$, 
which means that $T$ is simply linked. 
Moreover, we obviously have $R(T)=K$, and 
\[\lambda_a(T)=\lambda_a(E)=\omega_a(D)=0\] 
for any $a\in\{0,1\}$. 

Since the closure of $L(T)$ is presented by $D$ 
with the basepoint $q$, 
we have $\widehat{L(T)}=\widehat{K}$  
and $sg_1(L(T))\leq g=sg_1(K)$. 
\end{proof}

\begin{lemma}\label{lem53}
If a simply linked tangle $T$
stasifies $\lambda_0(T)=\lambda_1(T)=0$, 
then 
\[F_{ab}(L(T);t)=H_{a'b'}(R(T);t)\mbox{ and }
H_{ab}(L(T);t)=F_{a'b'}(R(T);t)\]
hold for any $a,b\in\{0,1\}$, where $a'=1-a$ and $b'=1-b$. 
\end{lemma}

\begin{proof}
Let $E$ be a diagram of $T$ with 
$J_0(E)=J_1(E)=\emptyset$. 
We regard the real crossings $c_1,\dots,c_m$ of $E$ 
as those of the right closure $R(E)$. 
Note that $R(E)$ is untwisted 
by $\lambda_0(E)=\lambda_1(E)=0$. 

For each $i$ with $1\leq i\leq m$, 
let $c_i'$ denote the crossing of the left closure $L(E)$ corresponding to 
$c_i$, 
and $\alpha_i'$ and $\beta_i'$ the cycles 
at $c_i'$ of $L(E)$. 
Then it holds that $\alpha_i'=\beta_i$ and $\beta_i'=\alpha_i$ 
for $1\leq i\leq m$, and 
\[I_a(L(E))=J_{a'}(E)=I_{a'}(R(E)).\]
Therefore, we have  
\begin{align*}
F_{ab}(L(T);t)=f_{ab}(L(E);t)
&=\sum_{i\in I_a(L(E)), j\in I_b(L(E))}\e_i\e_j
(t^{\alpha_i'\cdot\alpha_j'}-1)\\
&=\sum_{i\in I_{a'}(R(E)), j\in I_{b'}(R(E))}\e_i\e_j
(t^{\beta_i\cdot\beta_j}-1)\\
&=h_{a'b'}(R(E);t)=H_{a'b'}(R(T);t).
\end{align*}
Similarly, we have $H_{ab}(L(T);t)=F_{a'b'}(R(T);t)$. 
\end{proof}

\begin{proposition}\label{prop54}
For any long virtual knot $K$, 
there exists a long virtual knot $K'$ such that 
\begin{enumerate}
\item
$F_{ab}(K';t)=H_{a'b'}(K;t)$, 
\item
$H_{ab}(K';t)=F_{a'b'}(K;t)$, and 
\item
$sg_1(K')\leq sg_1(K)$ 
\end{enumerate}
for any $a,b\in\{0,1\}$, where $a'=1-a$ and $b'=1-b$. 
\end{proposition}

\begin{proof}
By Lemma~\ref{lem52}, 
we may choose a simply linked tangle $T$ 
with $R(T)=K$ and $\lambda_0(T)=\lambda_1(T)=0$. 
Let $K'=L(T)$. 
The conclusion then follows from Lemmas~\ref{lem52} and \ref{lem53}.
\end{proof}

\section{Characterization of the writhe polynomials}\label{sec6}

In \cite{HNNS2}, the writhe polynomials 
$W_0(K;t)$ and $W_1(K;t)$ are characterized as follows. 

\begin{lemma}[{\cite[Proposition 4.4]{HNNS2}}]\label{lem61}
Let $f(t)\in{\Z}[t,t^{-1}]$ be a Laurent polynomial. 
Then the following are equivalent. 
\begin{enumerate}
\item
There exists a long virtual knot $K$ with $W_0(K;t)=f(t)$. 
\item
$f(1)=0$. 
\end{enumerate}
The same equivalence holds for $W_1$. 
\qed
\end{lemma}

In this section, we prove the following, 
which is slightly stronger than Lemma~\ref{lem61}. 

\begin{proposition}\label{prop62}
Let $f(t)\in{\Z}[t,t^{-1}]$ be a Laurent polynomial. 
Then the following are equivalent. 
\begin{enumerate}
\item
There eixsts a long virtual knot $K$ with $W_0(K;t)=f(t)$ and $sg_2(K)=0$. 
\item
$f(1)=0$. 
\end{enumerate}
The same equivalence holds for $W_1$. 
\end{proposition}

We remark that, if $f(t)\ne 0$, 
then any long virtual knot $K$ in Proposition~\ref{prop62}(i) 
satisfies $sg_1(K)=1$; 
in fact, $K$ is nonclassical. 

For $n\geq 1$, 
let $J_n$ denote the long virtual knot 
presented by a diagram $D_n$ 
with $2n$ real crossings $c_1,\dots,c_{2n}$, 
as shown at the top of Figure~\ref{fig:characterization-w}. 

\begin{figure}[htbp]
  \centering
    \begin{overpic}[]{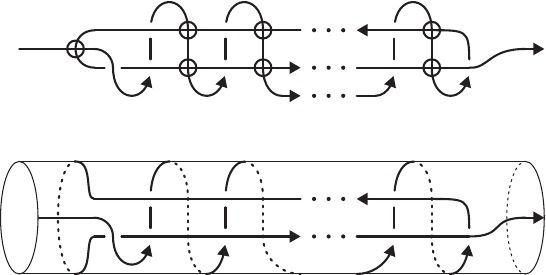}
      \put(61,124){$c_{3}$}
      \put(98,124){$c_{5}$}
      \put(165,124){$c_{2n-1}$}
      \put(45,92){$c_{1}$}
      \put(75.5,92){$c_{2}$}
      \put(112,92){$c_{4}$}
      \put(194,88){$\nwarrow$}
      \put(204,78){$c_{2n-2}$}
      \put(230,92){$c_{2n}$}
      \put(61,43){$c_{3}$}
      \put(98,43){$c_{5}$}
      \put(165,43){$c_{2n-1}$}
      \put(45,11){$c_{1}$}
      \put(75.5,11){$c_{2}$}
      \put(112,11){$c_{4}$}
      \put(193,11){$c_{2n-2}$}
      \put(230,11){$c_{2n}$}
    \end{overpic}
  \caption{The long virtual knot $J_n$}
  \label{fig:characterization-w}
\end{figure}
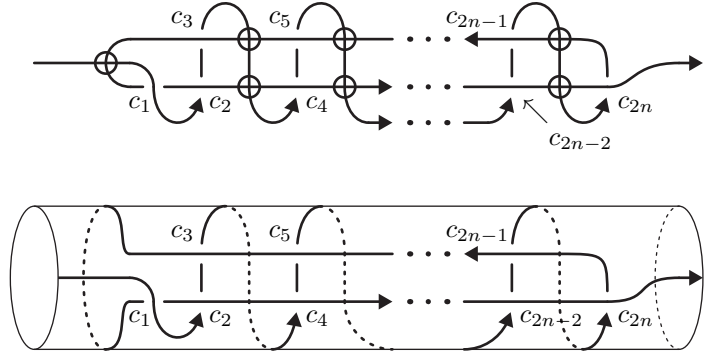

\begin{lemma}\label{lem63}
The long virtual knot $J_n$ $(n\geq 1)$ 
satisfies the following. 
\begin{enumerate}
\item $W_0(J_n;t)=W_1(J_n;t)=t^n-1$. 
\item $F_{ab}(J_n;t)=G_{ab}(J_n;t)=0$ for any $a,b\in\{0,1\}$. 
\item $H_{ab}(J_n;t)=-t^n+2-t^{-n}$ for any $a,b\in\{0,1\}$. 
\item $sg_1(J_n)=1$ and $sg_2(J_n)=0$. 
\end{enumerate}
\end{lemma}

\begin{proof}
Since the diagram $D_n$ has a surface realization 
on the annulus $\Sigma_{0,2}$ as shown at the bottom of Figure~\ref{fig:characterization-w}, 
we have $sg_2(J_n)=0$. 
Moreover, since it hols that 
\[I_0(D_n)=\{1\}, \ \e_1=1, \mbox{ and } \alpha_1\cdot\beta_1=n,\]
we have 
$W_0(J_n;t)=t^n-1$, and hence $sg_1(J_n)=1$. 
The remaining assertions follow immediately from Lemma~\ref{lem36}. 
\end{proof}

We define a set of Laurent polynomials by 
\[
\mathcal{P}=\{W_0(K;t)
\mid \mbox{$K$: a long virtual knot with $sg_{2}(K)=0$}\}. 
\]

\begin{lemma}\label{lem64}
For any Laurent polynomials $f(t)$ and $g(t)\in\mathcal{P}$, 
we have 
\begin{enumerate}
\item
$-f(t)\in\mathcal{P}$, 
\item 
$f(t^{-1})\in\mathcal{P}$, and 
\item 
$f(t)+g(t)\in\mathcal{P}$. 
\end{enumerate}
\end{lemma}

\begin{proof}
Let $K$ and $K'$ be long virtual knots such that 
\[W_{0}(K;t)=f(t), \ W_{0}(K';t)=g(t),\mbox{ and }
sg_{2}(K)=sg_{2}(K')=0.\]

(i) By Lemmas~\ref{lem22}(i) and \ref{lem36}(i), it holds that  
\[W_0(K^\#;t)=-W_1(K;t)=-W_0(K;t)=-f(t).\]
Since $sg_2(K^\#)=sg_2(K)=0$, 
we have $-f(t)\in\mathcal{P}$. 

(ii) By Lemmas~\ref{lem22}(i), (iii), and \ref{lem36}(i), 
it holds that  
\[W_{0}(K^{\#*};t)=W_{1}(K;t^{-1})=W_{0}(K;t^{-1})=f(t^{-1}). \]
Since $sg_2(K^{\#*})=sg_2(K)=0$, 
we have $f(t^{-1})\in\mathcal{P}$. 

(iii) 
By Lemma~\ref{lem24}(i), 
it holds that
\[W_0(K\circ K';t)=W_0(K;t)+W_0(K';t)=f(t)+g(t).\]
Since $sg_2(K\circ K')\leq sg_2(K)+sg_2(K')=0$, 
we have 
$f(t)+g(t)\in\mathcal{P}$. 
\end{proof}

\begin{proof}[Proof of {\rm Proposition~\ref{prop62}}.] 
The implication (i)$\Rightarrow$(ii) follows from $W_0(K;1)=0$. 
We prove (ii)$\Rightarrow$(i). 
Since $f(1)=0$, we can write 
\[f(t)=\sum_{k\ne 0}c_k(t^k-1)\]
for some integers $c_k$. 
Since $t^k-1\in\mathcal{P}$ $(k>0)$ by Lemma~\ref{lem63}(i), 
we have $f(t)\in\mathcal{P}$ by Lemma~\ref{lem64}. 

Since $W_0(K;t)=W_1(K;t)$ by Lemma~\ref{lem36}(i), 
the same characterization holds for $W_1$ as well. 
\end{proof}

\section{Characterization of 
the intersection polynomials}\label{sec7}

In this section, 
we address the realizability problem for the intersection polynomials of long virtual knots, 
and establish a complete characterization for each of 
the twelve polynomials.  
We first focus on the polynomials 
$F_{00}$, $F_{11}$, $H_{00}$, and $H_{11}$. 

\begin{theorem}\label{thm71}
Let $f(t)\in{\Z}[t,t^{-1}]$ be a Laurent polynomial. 
Then the following are equivalent. 
\begin{enumerate}
\item
There eixsts a long virtual knot $K$ 
with $F_{00}(K;t)=f(t)$. 
\item
There eixsts a long virtual knot $K'$ 
with $F_{00}(K';t)=f(t)$ and $sg_1(K')\leq 1$. 
\item
$f(1)=0$ and $f(t)=f(t^{-1})$. 
\end{enumerate}
The same equivalence holds for 
$F_{11}$, $H_{00}$, and $H_{11}$. 
\end{theorem}

\begin{proof}
The implication (ii)$\Rightarrow$(i) is trivial, 
and (i)$\Rightarrow$(iii) follows from Lemma~\ref{lem21}(ii). 

We prove (iii)$\Rightarrow$(ii). 
Since $f(1)=0$ and $f(t)=f(t^{-1})$, 
we can write 
\[f(t)=g(t)+g(t^{-1})\]
for some $g(t)\in{\Z}[t,t^{-1}]$ with $g(1)=0$; 
in fact, there are integers $c_k$ such that 
$f(t)=\sum_{k>0} c_k(t^k-2+t^{-k})$, 
and we may take $g(t)=\sum_{k>0}c_k(t^k-1)$. 
Since $g(1)=0$, 
there exists a long virtual knot $K$ such that 
$W_{0}(K;t)=g(t)$ and $sg_2(K)=0$ 
by Proposition~\ref{prop62}. 
Moreover, $K$ satisfies $F_{00}(K;t)=0$ 
by Lemma~\ref{lem36}(ii).

Let $T_1$ be the virtual tangle shown on the left of Figure~\ref{fig:t1}. 
Then we have 
\[F_{00}(R(T_1);t)=0 \mbox{ and } \lambda_0(T_1)=1.\]
Therefore, it follows from Theorem~\ref{thm43}(ii) that 
\[F_{00}(T_1+K)=F_{00}(K;t)+W_0(K;t)+W_0(K;t^{-1})
=0+g(t)+g(t^{-1})=f(t).\]
Moreover, since $sg_2(K)=0$, 
the right of the figure shows that $sg_1(T_1+K)\leq 1$. 

\begin{figure}[htbp]
  \centering
    \begin{overpic}[]{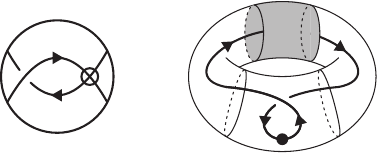}
      \put(-25.5,33){$T_{1}=$}
      \put(133,55){$K$}
      \put(132,-12){$T_{1}$}
    \end{overpic}
  \vspace{1em}
  \caption{The tangle $T_{1}$ and a surface realization of $T_{1}+K$}
  \label{fig:t1}
\end{figure}

For $F_{11}$, since it follows from Lemma~\ref{lem23}(i) that 
\begin{align*}
&F_{11}((T_1+K)^\#;t)=F_{00}(T_1+K;t)=f(t)
\mbox{ and }\\
&sg_1((T_1+K)^\#)=sg_1(T_1+K)\leq 1,
\end{align*}
the same characterization holds for $F_{11}$ as well. 
For $H_{00}$ and $H_{11}$, 
by Lemma~\ref{lem23}(i) and Proposition~\ref{prop54}(ii) and (iii), 
there is a long virtual knot $K'$ such that 
\begin{align*}
&H_{00}(K'^\#;t)=H_{11}(K';t)=F_{00}(T_1+K;t)=f(t)\mbox{ and}\\
&sg_1(K'^\#)=sg_1(K')\leq sg_1(T_1+K)\leq 1.
\end{align*} 
Therefore, the same characterization holds for 
$H_{00}$ and $H_{11}$. 
\end{proof}

We remark that, if $f(t)\ne 0$, then any long virtual knot $K'$ in Theorem~\ref{thm71}(ii) 
satisfies $sg_1(K')=1$; in fact, $K'$ is nonclassical. 

The next theorem gives a characterization 
for $G_{00}$ and $G_{11}$.

\begin{theorem}\label{thm72}
Let $f(t)\in{\Z}[t,t^{-1}]$ be a Laurent polynomial. 
Then the following are equivalent. 
\begin{enumerate}
\item
There eixsts a long virtual knot $K$ 
with $G_{00}(K;t)=f(t)$. 
\item
There eixsts a long virtual knot $K'$ 
with $G_{00}(K';t)=f(t)$ and $sg_1(K')\leq 1$. 
\item
$f(1)=f'(1)=0$. 
\end{enumerate}
The same equivalence holds for $G_{11}$. 
\end{theorem}

\begin{proof} 
The implication (ii)$\Rightarrow$(i) is trivial, 
and (i)$\Rightarrow$(iii) follows from Lemma~\ref{lem21}(iii). 

We prove (iii)$\Rightarrow$(ii). 
Since $f(1)=f'(1)=0$, 
we can write 
\[f(t)=(1-t^{-1})g(t)\] for some $g(t)\in{\Z}[t,t^{-1}]$ 
with $g(1)=0$. 
Then there exists a long virtual knot $K$ such that 
$W_{0}(K;t)=g(t)$ and $sg_2(K)=0$ 
by Proposition~\ref{prop62}. 
Moreover, $K$ satisfies $G_{00}(K;t)=0$ 
by Lemma~\ref{lem36}(ii). 

Let $T_2$ be the virtual tangle shown on the left of Figure~\ref{fig:t2}. 
Then we have 
\[G_{00}(R(T_2);t)=0, \ \lambda_0(T_2)=-1, 
\mbox{ and }V_0(T_2;t)=-t^{-1}.\]
Therefore, it follows from Theorem~\ref{thm43}(iii) that 
\[G_{00}(T_2+K)=G_{00}(K;t)+(1-t^{-1})W_0(K;t)=0+(1-t^{-1})g(t)=f(t).\]
Moreover, since $sg_2(K)=0$, 
the right of the figure shows that $sg_1(T_2+K)\leq 1$. 

\begin{figure}[htbp]
  \centering
    \begin{overpic}[]{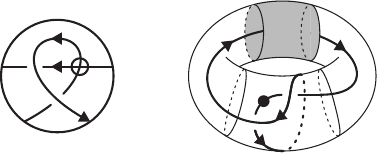}
      \put(-25.5,33){$T_{2}=$}
      \put(133,55){$K$}
      \put(132,-12){$T_{2}$}
    \end{overpic}
  \vspace{1em}
  \caption{The tangle $T_{2}$ and a surface realization of $T_{2}+K$}
  \label{fig:t2}
\end{figure}

For $G_{11}$, since it follows from Lemma~\ref{lem23}(i) that 
\begin{align*}
&G_{11}((T_2+K)^\#;t)=G_{00}(T_2+K;t)=f(t)
\mbox{ and }\\
&sg_1((T_2+K)^\#)=sg_1(T_2+K)\leq 1,
\end{align*}
the same characterization holds for $G_{11}$ as well. 
\end{proof}

\begin{remark}\label{rem73}
(i) In the proof of Theorem~\ref{thm72}, 
we have $G_{00}(R(T_2);t)=0$, 
which can be generalized as follows. 
Suppose that a long virtual knot $K$ has a diagram 
where a real crossing $c_1$ is of type $0$, 
and the others are of type $1$. 
Then we have  
\begin{align*}
F_{00}(K;t)&=\e_1^2(t^{\alpha_1\cdot\alpha_1}-1)=0 \mbox{ and}\\
G_{00}(K;t)&=
\e_1^2(t^{\alpha_1\cdot\beta_1}-1)
-\e_1\cdot \e_1(t^{\alpha_1\cdot\beta_1}-1)=0.
\end{align*}

\noindent
(ii) If $f(t)\ne 0$, then any knot $K'$ in Theorem~\ref{thm72}(ii) 
satisfies $sg_1(K')=1$. 

\noindent
(iii) For (closed) virtual knots, 
the same characterization as in Theorem~\ref{thm72} also holds 
for the writhe polynomial~\cite{ST} and 
for the first intersection polynomial~\cite{HNNS3}. 
\end{remark}

We conclude this paper by giving 
a characterization of the remaining intersection polynomials 
as follows. 

\begin{theorem}\label{thm74}
Let $f(t)\in{\Z}[t,t^{-1}]$ be a Laurent polynomial. 
Then the following are equivalent. 
\begin{enumerate}
\item
There eixsts a long virtual knot $K$ 
with $F_{01}(K;t)=f(t)$. 
\item
There eixsts a long virtual knot $K'$ 
with $F_{01}(K';t)=f(t)$ and $sg_1(K')\leq 1$. 
\item
$f(1)=0$. 
\end{enumerate}
The same equivalence holds for 
$F_{10}$, $G_{01}$, $G_{10}$, 
$H_{01}$ and $H_{10}$. 
\end{theorem}

\begin{proof}
The implications (ii)$\Rightarrow$(i) 
and (i)$\Rightarrow$(iii) are trivial. 

We first prove (iii)$\Rightarrow$(ii) for $F_{01}$ and $G_{01}$. 
Since $f(1)=0$, 
there exists a long virtual knot $K$ such that 
$W_{0}(K;t)=f(t)$ and $sg_2(K)=0$ 
by Proposition~\ref{prop62}. 
Moreover, $K$ satisfies $F_{01}(K;t)=G_{01}(K;t)=0$ 
by Lemma~\ref{lem36}(ii). 

Let $T_3$ and $T_4$ be the virtual tangles as shown in Figure~\ref{fig:t3t4}. 
Then we have 
\begin{align*}
& F_{01}(R(T_3);t)=0, \ \lambda_0(T_3)=0, \ \lambda_1(T_3)=1, \\
& G_{01}(R(T_4);t)=0, \ \lambda_1(T_4)=-1, \mbox{ and }
V_0(T_4;t)=0.
\end{align*}
Therefore, it follows from Theorem~\ref{thm43}(ii) and (iii) that 
\begin{align*}
& F_{01}(T_3+K;t)=F_{01}(K;t)+W_0(K;t)=0+f(t)=f(t) \mbox{ and}\\
& G_{01}(T_4+K;t)=G_{01}(K;t)+W_0(K;t)=0+f(t)=f(t). 
\end{align*}
Moreover, since $sg_2(K)=0$, we have 
$sg_1(T_i+K)\leq 1$ $(i=3,4)$. 

Similarly to the proof of Theorem~\ref{thm71}, 
the same characterization holds for 
the remaining polynomials $F_{10}$, $G_{10}$, $H_{01}$, and $H_{10}$. 
\end{proof} 

\begin{figure}[htbp]
  \centering
    \begin{overpic}[]{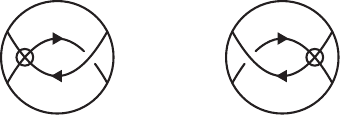}
      \put(24,-12){$T_{3}$}
      \put(132,-12){$T_{4}$}
    \end{overpic}
  \vspace{1em}
  \caption{The tangles $T_{3}$ and $T_{4}$}
  \label{fig:t3t4}
\end{figure}

We remark that, if $f(t)\ne 0$, then any long virtual knot $K'$ in Theorem~\ref{thm74}(ii) 
satisfies $sg_1(K')=1$. 



\end{document}